\documentclass[11pt]{amsart}
\usepackage{amsmath,amsthm,amsfonts,latexsym,amssymb}
\usepackage{graphicx}
\usepackage[arrow, matrix, curve]{xy}

\calclayout\evensidemargin .1in

\newcommand{\R}{\mathbb{R}}

\newtheorem{theorem}{Theorem}[section]
\newtheorem{lemma}[theorem]{Lemma}

\theoremstyle{definition}

\newtheorem{remark}[theorem]{Remark}

\newtheorem{definition}[theorem]{Definition}

\title[On Legendrian Embbeddings into Open Book Decompositions]{On Legendrian Embbeddings into Open Book Decompositions}

\author{Selman Akbulut}
\address{Department of Mathematics, Michigan State University, Lansing MI, USA}
\email{akbulut@math.msu.edu}
\thanks{The first author is partially supported by NSF FRG grant DMS- 0905917}

\author{M. Firat Arikan}
\address{Dept. of Mathematics, Middle East Technical University, Ankara, TURKEY}
\email{farikan@metu.edu.tr}
\thanks{The second author is partially supported by NSF FRG grant DMS-1065910, and also by TUBITAK grant 1109B321200181}

\date{\today}

\begin{document}

\begin{abstract}
We study Legendrian embeddings of a compact Legendrian submanifold $L$ sitting in a closed contact manifold $(M,\xi)$ whose contact structure is supported by a (contact) open book $\mathcal{OB}$ on $M$. We prove that if $\mathcal{OB}$ has  Weinstein pages, then there exist a contact structure $\xi'$ on $M$, isotopic to $\xi$ and supported by $\mathcal{OB}$, and a contactomorphism $f:(M,\xi) \to (M,\xi')$ such that the image $f(L)$ of any such submanifold can be Legendrian isotoped so that it becomes disjoint from the closure of a page of $\mathcal{OB}$.
\end{abstract}

\maketitle

%-----------------------------------------------------------------------------------------------
%-----------------------------------------------------------------------------------------------
%-----------------------------------------------------------------------------------------------

\section{Introduction}

A \emph{contact manifold} is a pair $(M^{2n+1},\xi)$ where $M$ is a smooth manifold and $\xi\subset TM$ is a totally non-integrable $2n$-plane field distrubution on $M$. The distrubution $\xi$ is called a \emph{contact structure} on $M$,  and is said to be \emph{co-oriented} if it is the kernel of a globally defined $1$-form $\alpha$ with the property $\alpha \wedge (d\alpha)^{n} \neq 0$. Such  a $1$-form is called a \emph{contact form} on $M$. Here we always assume that $\xi$ is a co-oriented \emph{positive} contact structure, that is, $\xi=\textrm{Ker}(\alpha)$ and $\alpha \wedge (d\alpha)^{n} > 0$ with respect to a pre-given orientation on $M$. We say that two contact manifolds $(M,\xi)$ and $(M',\xi')$ are \emph{contactomorphic} if there exists a diffeomorphism $f:M\longrightarrow M'$ such that $f_\ast(\xi)=\xi'$. Two contact structures $\xi_0, \xi_1$ on a $M$ are said to be \emph{isotopic} if there exists a $1$-parameter family $\xi_t$ ($0\leq t\leq 1$) of contact structures joining them.\\

A submanifold $L \subset (M^{2n+1},\xi)$ is said to be \emph{isotropic} if \,$TL \subset \xi$, and an isotropic submanifold is called \emph{Legendrian} if $\textrm{dim}(L)=n$. A \emph{Legendrian embedding} is an embedding $\phi:\Sigma^n \hookrightarrow (M^{2n+1},\xi)$ of a smooth manifold $\Sigma^n$ such that the image $\phi(\Sigma^n)$ is Legendrian. A smooth 1-parameter family of Legendrian submanifolds is called a \emph{Legendrian isotopy}. Equivalently, a Legendrian isotopy is a smooth 1-parameter family $\phi_t:\Sigma^n \hookrightarrow (M^{2n+1},\xi)$ of Legendrian embeddings.\\

An \emph{open book} (\emph{decomposition}) $\mathcal{OB}$ on a closed manifold $M$ is determined by a pair $(B,\varphi)$ where $B\hookrightarrow M$ is codimension $2$ submanifold  with trivial normal bundle, and $\varphi : M-B \rightarrow S^1$ is a fiber bundle projection. The neighborhood of $B$ should have a trivialization $B \times D^2$, where the angle coordinate on the disk agrees with the fibration map $\varphi$. The manifold $B$ is called the \emph{binding}, and for any $t_0 \in S^1$ a fiber $X = \varphi^{-1}(t_0)$ is called a \emph{page} of the open book. \\

The following definition is due to Giroux \cite{Gi}: A contact structure $\xi$ on $M$ is said to be \emph{supported by} (or \emph{carried by}, or \emph{compatible with}) an open book $\mathcal{OB}=(B, \varphi)$ on $M$ if there exists a contact form $\alpha$ for $\xi$ such that
\begin{itemize}
\item[(i)] $(B, \alpha |_{TB})$ is a contact manifold.
\item[(ii)] For any $t\in S^1$, the page $X=\varphi^{-1}(t)$ is a symplectic manifold with symplectic form $d\alpha$.
\item[(iii)] If $\bar{X}$ denotes the closure of a page $X$ in $M$, then the orientation of $B\cong\partial \bar{X}$ induced by its contact form $\alpha |_{TB}$ coincides with its orientation as the boundary of $(\bar{X}, d\alpha)$.
\end{itemize}
We will say that an open book $\mathcal{OB}$ on $M$ is called a \emph{contact open book} if it carries a contact structure on $M$.\\

In the smooth category, given an $n$-dimensional submanifold $L$ of a closed $(2n+1)$-manifold $M$ admitting an open book $\mathcal{OB}$, it is, in general, not possible to isotope $L$ so that it becomes disjoint from a page $\bar{X}$ of $\mathcal{OB}$. However, if the spine of $\bar{X}$ is $n$-dimensional, the general position argument shows that we can make them disjoint. Moreover, it is known (see below for details) that any Weinstein domain of dimension $2n$ is homotopy equivalent to its core which is an $n$-dimensional CW-complex. Here we prove:

\begin{theorem} \label{thm:Main_Theorem_1}
Let $(M,\xi)$ be a contact manifold which admits a contact open book $\mathcal{OB}$ supporting $\xi$. If the pages of $\mathcal{OB}$ are Weinstein, then there exist a contact structure $\xi'$ on $M$, which is isotopic to $\xi$ and supported by $\mathcal{OB}$, and a contactomorphism $f:(M,\xi) \to (M,\xi')$ such that the image $f(L)$ of any compact embedded Legendrian submanifold $L$ of $(M,\xi)$ can be Legendrian isotoped until it becomes disjoint from the closure of a page of $\mathcal{OB}$.
\end{theorem}

The proof of the theorem will be given in Section \ref{sec:Leg_Emb_into_Contact_OB}. One of the arguments used in the proof is making the given contact open book $\mathcal{OB}$ ``$\beta$-standard'' (see the next definition) by altering the compatible contact structure $\xi$ in its isotopy class.

\begin{definition} \label{def:Standard_contact_open_book}
 Let  $(X^{2n},d\beta)$ an exact symplectic manifold, and $M^{2n+1}$ be a manifold with a open book structure $\mathcal{OB}$ with pages $X$, supporting a contact structure $\xi$. Then we say that the $\mathcal{OB}$ is $\beta$-\emph{standard} if there exists a contact form $\alpha$ for $\xi$, such that $\alpha$ restricts to $\beta$ on every page.
\end{definition}

\medskip
We remark that any abstract contact open book is $\beta$-standard with respect to the contact form constructed by gluing the contact form $\beta+d\theta$ on the mapping torus $X \times \R / \sim$ with the one on $\partial X \times D^2$. This construction (which will be given in Remark \ref{rem:constructing_contact_form}) is originally due to Thurston and Winkelnkemper \cite{TW} for dimension three, and to Giroux \cite{Gi2} for higher dimensions. 
%\\

%==================================================================
\medskip \noindent {\em Acknowledgments.\/} The authors would like to
thank NSF and TUBITAK for their supports.
%==================================================================

\section{Preliminaries} \label{sec:Preliminaries}

A \emph{Liouville domain} (or a \emph{compact convex symplectic manifold}) is a pair $(W,\Lambda)$ where $W^{2n+2}$ is a compact manifold with boundary, together with a \emph{Liouville structure} (or a \emph{convex symplectic structure}), which meant there is a $1$-form $\Lambda$ on $W$ such that $\Omega=d\Lambda$ is symplectic and the $\Omega$-dual vector field $Z$ of $\Lambda$ defined by $\iota_Z \Omega=\Lambda$ (or, equivalently, $\mathcal{L}_{Z}\Omega=\Omega$), where $\iota$ denotes the interior product and $\mathcal{L}$ denotes the Lie derivative, should point strictly outwards along $\partial W$. Since $\Omega$ and $Z$ (resp. $\Omega$ and $\Lambda$) together uniquely determine $\Lambda$ (resp. $Z$), one can replace the notation with the triple $(W,\Omega,Z)$ (resp. $(W,\Omega,\Lambda)$). Here $Z$ is called {\it Lioville vector field}. The $1$-form $\Lambda_{\partial}:=\Lambda|_{\partial W}$ is contact (i.e., $\Lambda_{\partial} \wedge (d\Lambda_{\partial})^{n}>0$), and the contact manifold $(\partial W,\textrm{Ker} (\Lambda_{\partial}))$ is called the \emph{convex boundary} of $(W,\Lambda)$. \\

In the non-compact case (i.e., when $W$ is an open manifold), if we further assume that $Z$ is complete (i.e., its flow exists for all times), and also that there exists an exhaustion $W=\bigcup_{k=1}^{\infty} W^k$ by compact domains $W^k \subset W$ such that each $(W_k,\Lambda|_{W_k})$ is a Liouville domain with the convex boundary $(\partial W_k,\Lambda|_{\partial W_k})$ for all $k\geq 1$, then $(W,\Lambda,Z)$ is called a \emph{Liouville manifold} (see \cite{CE} for details). A \emph{Liouville cobordism} $(W,\Lambda, Z)$ is a compact cobordism $W$ with a Liouville structure $(\Lambda, Z)$ such that $Z$ points outwards along $\partial_+ W$ and inwards along $\partial_- W$. Note that any Liouville domain is a Liouville cobordism with $\partial_- W=\emptyset$.\\

If $Z^{-t}:W \rightarrow W$  $(t>0)$ denotes the contracting flow of $Z$, then the \emph{core} (or \emph{skeleton}) of the Liouville manifold $(W,\Lambda)$ is defined to be the set
$$\textrm{Core}(W,\Lambda):=\bigcup_{k=1}^{\infty} \bigcap_{t>0} Z^{-t}(W^k).$$ The interior of the core of any Liouville manifold is empty (Lemma 11.1, \cite{CE}), and so the core of any Liouville domain $(W,\Lambda, Z)$ is compact. If $M$ denotes the convex boundary $\partial W$, then one can see that the negative half of the symplectization $(M \times \mathbb{R}, d(e^t \Lambda |_{M}))$ symplectically embeds into $W$ (as a collar neighborhood of $M$ in $W$) so that its complement in $W$ is $\textrm{Core}(W,\Lambda)$ and the embedding matches the positive $t$-direction of $\mathbb{R}$ with $Z$.
The \emph{completion} $(\widehat{W},\widehat{\Lambda})$ of a Liouville domain $(W,\Lambda)$ is obtained from $W$ by gluing the positive part $\partial W \times [0,\infty)$ of the \emph{symplectization} $(\partial W \times \mathbb{R}, d(e^t \Lambda_{\partial}))$ of its convex boundary. Two Liouville domains $(W_i,\Lambda_i)$ are said to be \emph{Liouville isomorphic} if there exists a diffeomorphism $\phi:(\widehat{W}_0,\widehat{\Lambda}_0)\to (\widehat{W}_1,\widehat{\Lambda}_1)$, called a \emph{Liouville isomorphism}, between their completions $(\widehat{W}_i,\widehat{\Lambda}_i)$ such that    $\phi^*(\widehat{\Lambda}_1)=\widehat{\Lambda}_0+df$ where $f$ is some compactly supported smooth function on $\widehat{W}_0$.\\

To define Weinstein manifolds and domains, we need three preliminary definitions:

\begin{definition}
(i) A vector field $Z$ on a smooth manifold $W$ is said to \emph{gradient-like} for a smooth function $\Psi: W \to \R$ if $Z \cdot \Psi=\mathcal{L}_{Z} \Psi >0$ away from the critical point of $\Psi$.\\
(ii) A real-valued function is said to be \emph{exhausting} if it is proper and bounded from below.\\
(iii) An exhausting function $\Psi: W \to \R$ on a symplectic manifold $(W,\Omega)$ is said to be $\Omega$-\emph{convex} if there exists a complete Liouville vector field $Z$ which is gradient-like for $\Psi$.
\end{definition}

\begin{definition} \label{def:Weinstein_manifold}
A \emph{Weinstein manifold} $(W,\Omega,Z,\Psi)$ is a symplectic manifold $(W,\Omega)$ which admits a $\Omega$-convex Morse function $\Psi:W \to \R$ whose complete gradient-like Liouville vector field is $Z$. The triple $(\Omega,Z,\Psi)$ is called a \emph{Weinstein structure} on $W$. A \emph{Weinstein cobordism} $(W,\Omega,Z,\Psi)$ is a Liouville cobordism $(W,\Omega,Z)$ whose Liouville vector field $Z$ is gradient-like for a Morse function $\Psi:W \to \R$ which is constant on the boundary $\partial W$. A Weinstein cobordism with $\partial_- W=\emptyset$ is called \emph{Weinstein domain}.
\end{definition}

\begin{remark} \label{rem:Weinstein_manifold_is_Liouville}
Any Weinstein manifold $(W,\Omega,Z,\Psi)$ can be exhausted by Weinstein domains $W_k = \{\Psi^{-1}(-\infty,d_k] \} \subset W$ where $\{d_k\}$ is an increasing sequence of regular values of $\Psi$, and therefore, any Weinstein manifold is a Liouville manifold. In particular, any Weinstein domain is a Liouville domain. Also note that any Weinstein domain $(W,\Omega,Z,\Psi)$ has the convex boundary $(\partial W,\textrm{Ker}(\iota_{Z}\Omega|_{\,\partial W}))$, and the completion of a Weinstein domain is a Weinstein manifold.
\end{remark}

The following topologically characterizes Weinstein domains and will be used later.

\begin{theorem} [\cite{W}, see also Lemma 11.13 in \cite{CE}] \label{thm:Top_Charac_of_Weinstein_manifolds}
Any Weinstein domain of dimension $2n$ admits a handle decomposition whose handles have indices at most $n$.
\end{theorem}

%-----------------------------------------------------------------------------------------------
%-----------------------------------------------------------------------------------------------
%-----------------------------------------------------------------------------------------------

\vspace{.05in}
\section{Legendrian Embeddings into Contact Open Books} \label{sec:Leg_Emb_into_Contact_OB}

In this section we'll prove Theorem \ref{thm:Main_Theorem_1}. Let us start with by recalling how an open book $\mathcal{OB}=(B,\varphi)$ on a closed manifold $M$ determines an abtract open book $(X,h)$: By definition $B$ is a codimension-two subset of $M$ with a trivial normal bundle $B \times D^2$ and $\varphi:M\setminus B \to S^1$ is a fiber bundle map agreeing with the angular coordinate on the $D^2$-factor. To describe an abstract open book, we first pick a page $X=\varphi^{-1}(p)$ for some $p \in S^1$. Then the monodromy $h:X \to X$ can be read from the first return map of a vector field on $M$ transversal to the fibers of $\varphi$. Now using the resulting abstract open book we can construct a closed manifold $$M(X,h):=\Sigma(X,h) \cup_\partial (\partial X \times D^2)$$ where $\Sigma(X,h)$ is the mapping torus determined by $h$. Observe that the construction defines an open book decomposition $\mathcal{OB}_{(X,h)}$ on $M(X,h)$, and also that $M$ can be identified with $M(X,h)$ via some diffeomorphism $\Upsilon$ respecting the fibration maps on $M\setminus B$ and $M(X,h)\setminus (\partial X \times \{0\})$.\\

Next, from a collection of results from \cite{Gi} and \cite{Gi2}  (see also Section 7.3 of \cite{Geiges} for a detailed explanation) we recall the construction of a contact form (under suitable assumptions) on the manifold $M(X,h)$ in the following remark.

\begin{remark} \label{rem:constructing_contact_form}
Given an abstract open book $(X,h)$, consider the closed manifold $M(X,h)=\Sigma(X,h) \cup_\partial (\partial X \times D^2)$. Suppose that there exists a Liouville form $\beta$ on $X$ and $h \in \textrm{Symp}(X,d\beta)$. Then one can construct a contact structure $\xi_\beta$ on $M(X,h)$ so that $\mathcal{OB}_{(X,h)}$ becomes $\beta$-standard. Here we recall the explicit construction of the contact form $\alpha_\beta$ defining $\xi_\beta$: Since $h \in \textrm{Symp}(X,d\beta)$, the form $h^*(\beta)-\beta$ is closed, and it can be made exact by deforming $h$ through symplectomorphisms which are identity near $\partial X$.  Such a deformation of $h$ changes $M(X,h)$ in its diffeomorphism class, and so we may assume that $h^*(\beta)-\beta=-d\rho$ for some smooth function $\rho:X \to \R$. Adding a large enough constant, one can assume that $\rho$ is strictly positive everywhere on $X$, and so we can use $\rho$ to construct a smooth mapping torus $$\Sigma(X,h)_\rho:=X \times \R / \sim_{\rho} \quad \textrm{where} \quad (x,z) \sim_{\rho} (h(x),z+\rho(x)\,).$$ Consider the contact form $\alpha=\beta + \,dz$ on $X \times \R$, and let $\sigma_\rho:X \times \R \to X \times \R$ be the diffeomorphism defining $\sim_\rho$, that is, $\sigma_\rho(x,z)=(h(x),z+\rho(x)\,)$. Then we compute $$\sigma_\rho^*(\beta+dz)=h^*(\beta)+d(z+\rho)=\beta-d\rho+dz+d\rho=\beta+dz$$ which shows that $\alpha=\beta+dz$ decends to a contact form, say $\tilde{\alpha}$, on $\Sigma(X,h)_\rho$. Then using appropriate cut-off functions, one can construct a contact form $\alpha_\beta$ on $$\Sigma(X,h)_\rho \cup_\partial (\partial X \times D^2)\approx M(X,h)$$ by smoothly gluing $\tilde{\alpha}$ with the contact form $\beta|_{\partial X}+\frac{1}{2} r^{2} d\theta$ on the normal bundle $\partial X \times D^2$.
\end{remark}

The key observation for the proof of Theorem \ref{thm:Main_Theorem_1} is the following:

\begin{theorem} \label{thm:Missing_a_page}
Let $X$ be a Weinstein domain of dimension $2n \geq 2$ whose underlying Liouville structure is given by the Liouville form $\beta$. For $h \in \emph{Symp}(X,d\beta)$, consider the contact manifold  $(M(X,h),\xi_\beta:=\emph{Ker}(\alpha_\beta))$ as in Remark \ref{rem:constructing_contact_form}. Let $\phi: \Sigma^n \hookrightarrow (M(X,h),\xi_\beta)$ be a Legendrian embedding of a compact Legendrian submanifold $S:=\phi(\Sigma)\subset M(X,h)$. Then $S$ can be Legendrian isotoped (through Legendrian embeddings) to another embedded Legendrian submanifold $S'$ which  is disjoint from the closure of a page of $\mathcal{OB}_{(X,h)}$ on $M(X,h)$.
\end{theorem}

In order to prove this theorem we will make use of contact vector fields. A vector field on a contact manifold is said to be \emph{contact} if its flow preserves the contact distrubution. The following fundamental lemma in contact geometry characterizes contact vector fields on a given contact manifold. More details can be found, for instance, in \cite{Geiges}.

\begin{lemma} \label{lem:contact_vector_field}
Let $(M,\emph{Ker}(\alpha))$ be any contact manifold and $R_{\alpha}$ denote the Reeb vector field of $\alpha$. Then there is a one-to-one correspondence between the set $\{Z \in \Gamma(M)\,|\, Z \, \emph{ is contact}\}$ of all contact vector fields on $M$ and the set $\{H : M \to\R\,|\, H\, \emph{ is smooth}\}$ of all smooth functions on $M$. The correspondence is given by $Z \to H_Z := \alpha(Z)$ ($H_Z$ is called the ``contact Hamiltonian'' of the contact vector field $Z$), and $H \to Z_H$ where $Z_H$ is the contact vector field uniquely determined by the equations  $\alpha(Z_H)=H$ and $\iota_{Z_H} d\alpha=dH(R_{\alpha})\alpha-dH$. \qed
\end{lemma}

We note that a similar statement also holds between locally defined contact vector fields and locally defined smooth functons.

\begin{proof}[Proof of Theorem \ref{thm:Missing_a_page}]
For simplicity we will write $M(X,h)=Y$ and $\alpha_\beta=\alpha$.  We may assume that Legendrian submanifold $S=\phi(\Sigma^n) \subset (Y,\textrm{Ker}(\alpha))$ is connected since all the argument used below can be adapted (or generalized) to the case where $S$ has more than one connected component. So we have a Legendrian embedding of a compact connected manifold $\Sigma$ in the open book $\mathcal{OB}_{(X,h)}=(B,\varphi)$ on $Y$ corresponding to the abstract open book $(X,h)$. So we have
$$\phi: \Sigma^{n}\hookrightarrow Y=Y_1\cup Y_2=\Sigma(X,h)_\rho \cup (B \times D^2)$$
where $B=B \times \{0\} \subset B \times D^2 \subset Y$ is the binding, $Y_2:=B \times D^2$, and $Y_1:=\Sigma(X,h)_\rho$ is the mapping torus as above. We may consider $(Y_2,\textrm{Ker}(\alpha|_{Y_2}))$ as the contact manifold $$(B \times D^2,\textrm{Ker}(\beta+(x/2)\,dy -(y/2)\,dx))$$ where $(x,y)$ are the coordinates on the unit disk $D^2$.
Consider the following vector fields:
\begin{equation} \label{eqn:Contact_fields}
Z_1=\partial x-(y/2)R_{\beta},  \quad Z_2=\partial y +(x/2)R_{\beta},
\quad Z_{3}=\chi+ \theta \; \partial \theta, \quad Z_4=R_{\alpha}
\end{equation}
where $\chi$ is the $d\beta$-dual vector field of $\beta$, $R_{\alpha}$ (resp. $R_{\beta}$) denotes the Reeb vector field of $\alpha$ (resp. $\beta|_B$), and $\theta$ is the $S^1$-coordinate in the fibration $\varphi: Y \setminus B \to S^1$ determined by $\mathcal{OB}_{(X,h)}$. Here $\partial x=\partial / \partial x, \partial y=\partial / \partial y$, and so on.... Note that the first two are defined on $Y_2=B \times D^2$, and the third is defined on the contact manifold $(Y_1\setminus X,\textrm{Ker}(\alpha|_{Y_1\setminus X}))$ where $X$ is any fixed page of $\mathcal{OB}_{(X,h)}$. Here we note that $\alpha$ restricts to $\beta$ on every page of  $\mathcal{OB}_{(X,h)}$ (from its construction described in Remark \ref{rem:constructing_contact_form}) and $Y_1\setminus X=X \times \R$, and so, in particular, we may also write $Z_{3}=\chi+ z \; \partial z$ where $z$ is the $\R$-coordinate on $X \times \R$. It is easy to check that:
\begin{equation*} \mathcal{L}_{Z_{i}}\alpha = \left\{
\begin{array}{rl} 0 & \text{if } i=1,2,4\\
\alpha & \text{if } i=3 \end{array} \right.
\end{equation*}
So they are all contact vector fields on the regions where they are defined. In fact, if $H_i$ denotes the contact Hamiltonian function corresponding to $Z_i$ as in Lemma \ref{lem:contact_vector_field}, then we have
 $$H_1=-y, \quad H_2=x, \quad H_3=\theta, \quad H_4\equiv1.$$

We first see that one can make $S$ transverse to the binding $B$:

\begin{lemma} \label{lem:making_transverse_to_binding}
In any pre-given $\epsilon$-neighbourhood of $S$, one can Legendrian isotope $S$ (in a small neighbourhood of the binding $B$) so that it becomes everywhere transverse to $B$.
\end{lemma}

\begin{proof}
Let $N_\epsilon$ be any $\epsilon$-neighbourhood of $S$. We will use the fact that Legendrian (more generally, isotropic) submanifolds stays Legendrian (isotropic) under the flows of contact vector fields. Note that for any constants $a_1, a_2 \in \R$ the vector field $Z=a_1Z_1+a_2Z_2$ is contact with the
contact Hamiltonian $H_Z=a_1H_1+a_2H_2$. If $S$ intersects $B$ transversally, then there is nothing to prove. If not, let $K \subset S \cap (B \times \{(0,0)\})$ be the region where they don't intersect transversally. Consider $Z_1|_{B \times \{(0,0)\}}=\partial x,  Z_2|_{B \times \{(0,0)\}}=\partial y$ on $B=B \times \{(0,0)\}$. For any $p \in K$, the tangent space $$(\{\textbf{0}\}\times TD^2)|_{p} \subset (TB\times TD^2)|_{p}$$ does not lie in $TS|_{p}$ (otherwise $S$ and $B$ would intersect transversally at $p \in K$). Therefore, there exists a vector $v=a_1\partial x+a_2 \partial y \in (\{\textbf{0}\}\times TD^2)|_{B \times \{(0,0)\}}$ (for some constants $a_1, a_2 \in \R$) which is everywhere transverse to $S \cap B$. We consider $Z=a_1Z_1+a_2Z_2$ as the smooth extension of $v$ to the whole $B \times D^2$. Note that $Z$ will stay transverse to $S$ in a small neighbourhood $N_\delta:=B \times \{(x,y) \in D^2\, |\, x^2+y^2 < \delta\}$ for some $\delta >0$. Let $N_SK \subset N_\delta$ be a small neighbourhood of $K$ in $S$. Now
choose a regular value  $q\in \{(x,y) \in D^2\, |\, x^2+y^2 < \delta\} \subset D^2$ of the composition
$$\Sigma^n \supset U\stackrel{\phi }{\longrightarrow} B \times D^2 \stackrel{\pi_{2}}{\longrightarrow} D^2,\;\;\text{where} \;\phi (\Sigma)=S$$
where $\pi_{2}$ is the projection and $U=\phi^{-1}(B \times \textrm{int} (D^2))$, such that $q \in \pi_2(N_SK)$ lies on the line segment joining $(0,0)$ and $(a_1,a_2)$, and the above composition has no critical value other than $(0,0)$ on the line segment, say $l_q$, joining $(0,0)$ and $q$. Note that, by construction, $S$ intersects the identical copy $B_q=\pi_2^{-1}(q)$ of $B$ transversally, and $Z$ is everywhere transverse to $\pi_2^{-1}(l_q)$. Let $\mu: B \times D^2 \to \R$ be a cut-off function such that $\mu \equiv 1$ on a neigbourhood $N \subset N_\epsilon$ of $\pi_2^{-1}(l_q)$ in $N_\delta$, and $\mu \equiv 0$ on outside of a slightly larger neigbourhood $ N_\epsilon \subset N'$ (see Figure \ref{fig:making_transverse_to_binding}).\\

\begin{figure}[ht]
\begin{center}
\includegraphics{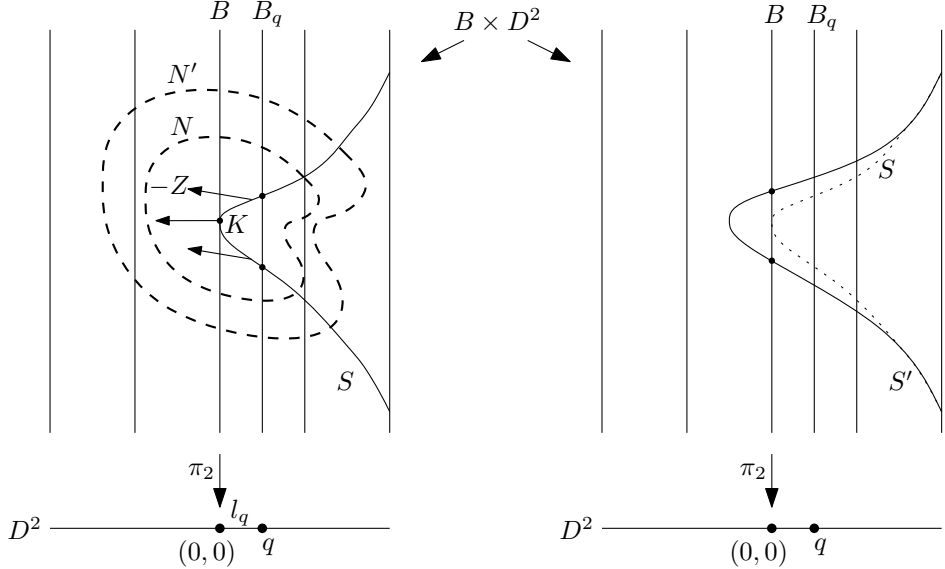}
\caption{Isotoping $S$ to another Legendrian submanifold $S'$ which is transverse to the binding $B=B \times \{(0,0)\}$.}
\label{fig:making_transverse_to_binding}
\end{center}
\end{figure}

Next consider the contact vector field $Z_{\mu}$ corresponding to the contact Hamiltonian $H_{\mu}:=\mu H_Z$. By the choice of $\mu$, $Z_{\mu}$ agrees with $Z$ on $N$ and it is identically zero outside $N'$. Using the backward flow $Z_\mu^{-t}$ of $Z_\mu$ we define the following $1$-parameter smooth family:
$$\Phi_t:\Sigma^n \longrightarrow Y, \quad \Phi_t(p)=Z_\mu^{-t}(\phi(p)), \quad t \in [0,T]$$ where $T$ is the time elapsed during the points of $S \cap B_q$ are moved to their final images in $B=B \times \{(0,0)\}$ under the backward flow of $Z_\mu$. (Note that all the points of $S \cap B_q \subset N$ reach $B$ at the same time because the ``horizontal'' components of $Z_\mu|_N=Z|_N$ is defined by the constants $a_1,a_2$.) Observe that $\Phi_0=\phi$, for each $t \in [0,T]$ we have $\Phi_t=\phi$ outside $N'$ and $\Phi_t(\Sigma^n)$ is Legendrian, and $S':=\Phi_1(\Sigma^n)\subset Y$ is everywhere transverse to the binding $B$ as depicted in Figure \ref{fig:making_transverse_to_binding}. Finally, by choosing $\delta$ small enough, one can guarantee that the isotopy $\Phi_t$ stays in the pre-given $\epsilon$-neighbourhood $N_\epsilon$ of $S$.
\end{proof}

By the above lemma we may assume that the Legendrian submanifold $S=\phi(\Sigma^n)$ is transverse to the binding $B$. Next, by picking a regular value $p\in S^1$ for the projection $\pi_{2}:\phi(\Sigma^n) \cap \Sigma(X,h)_\rho \to S^1$  we can assume that for the page $\varphi^{-1}(p) \approx X$ of $\mathcal{OB}_{(X,h)}$, the intersection 
$$L:=\phi (\Sigma^n) \cap \varphi^{-1}(p)= S\cap X $$ is a  properly imbedded  $n-1$ dimensional submanifold $(L, \partial L)\subset (X, B=\partial X)$ meeting the binding along an $n-2$ dimensional submanifold $\partial L$ (for simplicity we will write $X$ for $\varphi^{-1}(p) $). %As for the previous lemma, we point out that the next result holds not only for our case but it is also true in a more general situation. In other words, it is true for any Legendrian submanifold intersecting transversally a particular Weinstein page of a supporting open book.

\begin{lemma}\label{core} \label{lem:making_disjoint_from_core}
In any pre-given $\epsilon$-neighbourhood of $S$, one can Legendrian isotope $S$ (in a small neighbourhood of the page $X$) so that $L=S \cap X$  becomes disjoint from $\Delta=\emph{Core}(X,\beta)$.
\end{lemma}

\begin{proof}
Since $X^{2n}$ is Weinstein (by assumption), $\textrm{dim}(\Delta)=n$ by Theorem \ref{thm:Top_Charac_of_Weinstein_manifolds}. Also  we have $\textrm{dim}(\Delta)+\textrm{dim}(L^{n-1})=2n-1<2n=\textrm{dim}(X)$. Hence, by the general position in $X$, we can (topologically) isotope $L^{n-1}$ to a nearby copy which is disjoint from $\Delta$. This means that there exists a vector field $Z$ on $X$ which is transverse to both $L$ and $\Delta$ along their intersection $L \cap \Delta$. In what follows, using contact vector fields which are compactly supported near $L \cap \Delta $ (and which are generated from $Z$), we will construct an isotopy which transforms $L$ to some nearby copy $L'$ (disjoint from $\Delta$), and recognize this isotopy (in $X$) as the restriction of a local Legendrian isotopy (in $Y$) moving $S$ to another Legendrian submanifold, say $S'$.\\

Recall that there is a canonical contact model (Legendrian Neighbourhood Theorem) for the tubular neighbourhood $N_Y(S)$ of  $S$ in $Y$. That is, there exists a contactomorphism
$$\Upsilon:(T^*\Sigma^n \times \R, \textrm{Ker}(\textbf{q}d \textbf{p}+dz)) \longrightarrow (N_Y(S),\textrm{Ker}(\bar{\beta}^{st}|_{N_Y(S)}))$$
from the $1$-jet bundle $(T^*\Sigma^n \times \R, \textrm{Ker}(\textbf{q}d \textbf{p}+dz))$ where $\textbf{p}=(p_1,...,p_n), \textbf{q}=(q_1,...,q_n)$ are the standard coordinates on $T^*\Sigma^n$ and $z$ is the real coordinate. (Here $\Upsilon$ maps the zero section $\Sigma_0=\{\textbf{q}=\textbf{0}\} \times \{0\} \subset T^*\Sigma^n \times \R$ onto $S$.) Observe that, on $T^*\Sigma^n \times \R$, there are $2n+1$ linearly independent contact vector fields:
\begin{equation} \label{eqn:Contact_vector_fields_in_1_jet}
Z'_1=\partial p_1,...,Z'_n=\partial p_n, \quad Z'_{n+1}=\partial q_1-p_1 \partial z,...,Z'_{2n}=\partial q_n-p_n \partial z, \quad Z'_{2n+1}=\partial z
\end{equation}
The corresponding contact Hamiltonian functions (as in Lemma \ref{lem:contact_vector_field}), respectively, are
\begin{equation} \label{eqn:Contact_Hamiltonians_in_1_jet}
H'_1=q_1,...,H'_n=q_n, \quad H'_{n+1}=-p_1,...,H'_{2n}=-p_n, \quad H'_{2n+1}=1.
\end{equation}
We will use these contact vector fields for local Legendrian isotopies in $N_Y(S) \cong T^*\Sigma^n \times \R$ that we need for our purpose.\\

Let $L \cap \Delta=\sqcup_{i=1}^s L_i$ where $L_i$'s are (disjoint) connected components. Note that $L$ is compact as both $S=\phi(\Sigma^n)$ and $X$ are compact. Moreover, the core $\Delta$ is compact,  and so $L \cap \Delta$ is also compact from which we conclude that each $L_i $ is a compact CW-complex of finite type. Denote by $L_i^{K}$ the $K$-skeleton of $L_i$ for $K=0,1,...,n-1$. In particular, we have $L_i=L_i^{n-1}$ (Figure \ref{fig:making_disjoint_from_core}-a). Using the isotopies mentioned above, we will first make the closure of every $(n-1)$-cell in each $L_i$ disjoint from $\Delta$, and then do the same for $(n-2)$-cells, and so on... \\

Let $\{E_{i,1}^{K}, ..., E_{i,l_i^{K}}^{K}\}$ be the set of all $K$-cells in $L_i^{K}$. For any $1\leq j \leq l_i^{K}$, consider the following open cover for the closure $\overline{E_{i,j}^{K}}$ (recall the vector field $Z$ in the proof of Lemma~\ref{core}):
$$\mathcal{U}_{i,j}^{K}:=\{U_{x} \; |\; x \in E_{i,j}^{K}, U_{x} \textrm{ is a nbhd of } x \textrm{ in } \overline{E_{i,j}^{K}} \textrm{ s.t. } Z(x) \pitchfork \overline{U_{x}}\}. $$
Clearly, $\mathcal{U}_{i,j}^{K}$ covers $\overline{E_{i,j}^{K}}$ which is compact.
So, $\mathcal{U}_ {i,j}^{K}$ has a finite subcover $$\{U_{x^{K}_{i,j,1}}, U_{x^{K}_{i,j,2}},..., U_{x^{K}_{i,j,m_{i,j}^{K}}} \}$$
for some finite number of points $\{x^{K}_{i,j,1}, x^{K}_{i,j,2}, ..., x^{K}_{i,j,m_{i,j}^{K}}\}$ in $E_{i,j}^{K}$. We label these points in such a way that the neighbourhood of any point has nonempty intersection with the union of the neighbourhoods of the preceding points in the list as depicted in Figure \ref{fig:making_disjoint_from_core}-b (for the case $K=n-1$).\\

\begin{figure}[ht]
\begin{center}
\includegraphics{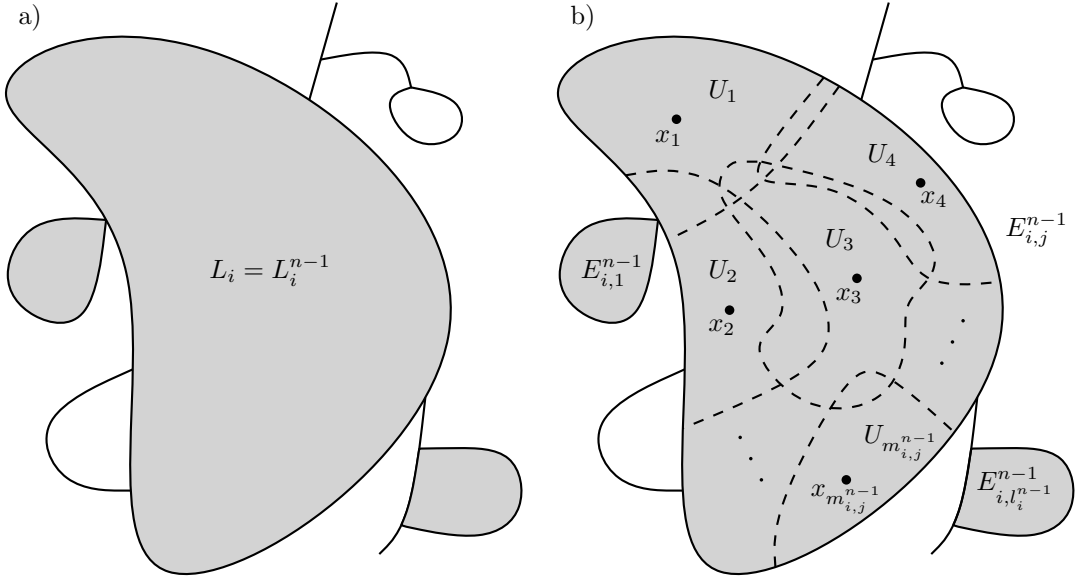}
\caption{a) A typical connected component $L_i=L_i^{n-1}$ of $L \cap \Delta$, b) Finite subcover of $\mathcal{U}_{i,j}^{n-1}$ (where we write $x^{n-1}_{i,j,1}=x_1,...,x^{n-1}_{i,j,m_{i,j}^{n-1}}=x_{m_{i,j}^{n-1}},$ and similarly $U_{x^{n-1}_{i,j,1}}=U_1,..., U_{x^{n-1}_{i,j,m_{i,j}^{n-1}}}=U_{m_{i,j}^{n-1}}$ for simplicity).}
\label{fig:making_disjoint_from_core}
\end{center}
\end{figure}

We will first make $\overline{U_{x^{n-1}_{i,j,1}}}$ disjoint from $\Delta$. From the definition of $\mathcal{U}_{i,j}^{n-1}$, we have a constant vector field $Z(x^{n-1}_{i,j,1}) \in TX|_{\overline{U_{x^{n-1}_{i,j,1}}}}$ which is everywhere transverse to $\overline{U_{x^{n-1}_{i,j,1}}}$. (Here at every point of $\overline{U_{x^{n-1}_{i,j,1}}}$, we have the same vector $Z(x^{n-1}_{i,j,1}) $.) Observe that, since $\{Z'_i\}_{i=1}^{2n+1}$ is a linearly independent set, the pull-back vector field $\Upsilon^*(Z(x^{n-1}_{i,j,1}))$ can be written as the unique linear combination
$$\Upsilon^*(Z(x^{n-1}_{i,j,1}))=c_1Z'_1+\cdots c_{2n+1}Z'_{2n+1}$$
for some unique constants $c_1,...,c_{2n+1} \in \R$. Since these constants depend on the point $x^{n-1}_{i,j,1}$, we set the notation $$\textbf{Z}_{x^{n-1}_{i,j,1}}:=c_1Z'_1+\cdots c_{2n+1}Z'_{2n+1}.$$ Let us write $V_{x^{n-1}_{i,j,1}}$ for the pre-image $\Upsilon^{-1}(U_{x^{n-1}_{i,j,1}})$.  Observe that $\textbf{Z}_{x^{n-1}_{i,j,1}}$ is a contact vector field on $T^*\Sigma^n \times \R$ with the contact Hamiltonian  $$\textbf{H}_{x^{n-1}_{i,j,1}}:=c_1H'_1+\cdots c_{2n+1}H'_{2n+1},$$ and also that it is everywhere transverse to $\overline{V_{x^{n-1}_{i,j,1}}}=\Upsilon^{-1}(\overline{U_{x^{n-1}_{i,j,1}}}) \subset \Upsilon^{-1}(L) \subset \Sigma_0$.

Denote by $\widetilde{V_{x^{n-1}_{i,j,1}}}$ a small neighbourhood of $\overline{V_{x^{n-1}_{i,j,1}}}$ in $\Upsilon^{-1}(N_Y(S) \cap X) \subset T^*\Sigma^n \times \R$. Let $\mu_{x^{n-1}_{i,j,1}}:\Upsilon^{-1}(N_Y(S) \cap X) \to \R$ be a smooth cut-off function such that $\mu_{x^{n-1}_{i,j,1}}\equiv 1$ near $\overline{V_{x^{n-1}_{i,j,1}}}$,
and $\mu_{x^{n-1}_{i,j,1}}\equiv 0$ on the complement $\Upsilon^{-1}(N_Y(S) \cap X) \setminus \widetilde{V_{x^{n-1}_{i,j,1}}}$. Now consider the contact vector field $\textbf{Z}_{\mu_{x^{n-1}_{i,j,1}}}$ whose corresponding contact Hamiltonian is equal to $\mu_{x^{n-1}_{i,j,1}} \textbf{H}_{x^{n-1}_{i,j,1}}$. By the choice of the cut-off function $\textbf{Z}_{\mu_{x^{n-1}_{i,j,1}}}|_{\overline{V_{x^{n-1}_{i,j,1}}}}=\textbf{Z}_{x^{n-1}_{i,j,1}}$, and so $\textbf{Z}_{\mu_{x^{n-1}_{i,j,1}}}$ is also transverse to $\overline{V_{x^{n-1}_{i,j,1}}}$. Using the flow $\textbf{Z}^t_{\mu_{x^{n-1}_{i,j,1}}}$ we isotope $\overline{V_{x^{n-1}_{i,j,1}}}$ to its nearby copy $\textbf{Z}^t_{\mu_{x^{n-1}_{i,j,1}}}(\overline{V_{x^{n-1}_{i,j,1}}})$ for some fixed time $t$. Note that pushing along a transverse contact vector field implies that $\textbf{Z}^t_{\mu_{x^{n-1}_{i,j,1}}}(\overline{V_{x^{n-1}_{i,j,1}}})$ is disjoint from $\Delta$, and is still isotropic in $(T^*\Sigma^n \times \R, \textrm{Ker}(\textbf{q}d \textbf{p}+dz))$. Indeed, since the transversality is an open condition, we know that the isotropic image $\textbf{Z}^t_{\mu_{x^{n-1}_{i,j,1}}}(\widehat{V_{x^{n-1}_{i,j,1}}})$ is  disjoint from $\Delta$ where $\widehat{V_{x^{n-1}_{i,j,1}}} \supset \overline{V_{x^{n-1}_{i,j,1}}}$ is a neighbourhood of $\overline{V_{x^{n-1}_{i,j,1}}}$ in $L_i$ such that 
$$\overline{V_{x^{n-1}_{i,j,1}}}\subset \widehat{V_{x^{n-1}_{i,j,1}}} \subset \widetilde{V_{x^{n-1}_{i,j,1}}} \cap L_i.$$

Similarly, we can make the closure of all the other open sets in the above finite subcover of $\mathcal{U}_{i,j}^{n-1}$ disjoint  from $\Delta$ (this will isotope the whole closed $(n-1)$-cell $\overline{E_{i,j}^{n-1}}$ to some isotropic copy which is disjoint from $\Delta$). However, for each such closure, the choice of how much we push it (using the flow of the corresponding contact vector field) needs a litle bit of more care: Let us discuss this in an inductive way: Suppose that we have already isotoped the union $$\overline{V_{x^{n-1}_{i,j,1}}} \cup \overline{V_{x^{n-1}_{i,j,2}}} \cup \cdots \cup \overline{V_{x^{n-1}_{i,j,k}}}, \quad \textrm{for some } k \in \{1,...,m_{i,j}^{n-1}-1\}$$
along the contact vector fields $\{\textbf{Z}_{\mu_{x^{n-1}_{i,j,m}}}\}_{m=1}^k$ (where the smooth cut-off functions $\mu_{x^{n-1}_{i,j,m}}:\Upsilon^{-1}(N_Y(S) \cap X) \to \R$ are constructed in the same way as above) so that the image of the union $$\bigcup_{m=1}^k \widehat{V_{x^{n-1}_{i,j,m}}}$$ is isotropic in $(T^*\Sigma^n \times \R, \textrm{Ker}(\textbf{q}d \textbf{p}+dz))$ and is  disjoint from $\Delta$ where $\widehat{V_{x^{n-1}_{i,j,m}}} \subset \widetilde{V_{x^{n-1}_{i,j,m}}} \cap L_i$ is a small neighbourhood of $\overline{V_{x^{n-1}_{i,j,m}}}$ in $L_i$. Now we would like to push (i.e., isotope) $\overline{V_{x^{n-1}_{i,j,k+1}}}$ using $\textbf{Z}_{\mu_{x^{n-1}_{i,j,k+1}}}$. Observe that the region $$\overline{V_{x^{n-1}_{i,j,k+1}}} \cap (\widehat{V_{x^{n-1}_{i,j,1}}} \cup \cdots \cup \widehat{V_{x^{n-1}_{i,j,k}}})$$
has been already made disjoint from $\Delta$, and also that $\textbf{Z}_{\mu_{x^{n-1}_{i,j,k+1}}}$ might be tangent to the image of this region at some points, or even its flow might transform some points in the region back to $\Delta$ (if we let them flow too much). On the other hand, since transversality is an open condition there exists a codim-$0$ subset $\overline{\overline{V_{x^{n-1}_{i,j,k+1}}}} \subset \overline{V_{x^{n-1}_{i,j,k+1}}}$ with a codim-$0$ nonempty intersection $$\overline{\overline{V_{x^{n-1}_{i,j,k+1}}}} \cap (\widehat{V_{x^{n-1}_{i,j,1}}} \cup \cdots \cup \widehat{V_{x^{n-1}_{i,j,k}}})$$ to which $\textbf{Z}_{\mu_{x^{n-1}_{i,j,k+1}}}$ is everywhere transverse. Therefore, we can make the union $$\overline{V_{x^{n-1}_{i,j,1}}} \cup \cdots \cup \overline{V_{x^{n-1}_{i,j,k+1}}}$$ disjoint from $\Delta$ by pushing (in an appropriate amount) along $\textbf{Z}_{\mu_{x^{n-1}_{i,j,k+1}}}$ as shown in Figure \ref{fig:pushing_keeping_disjoint}.

\begin{figure}[ht]
\begin{center}
\includegraphics{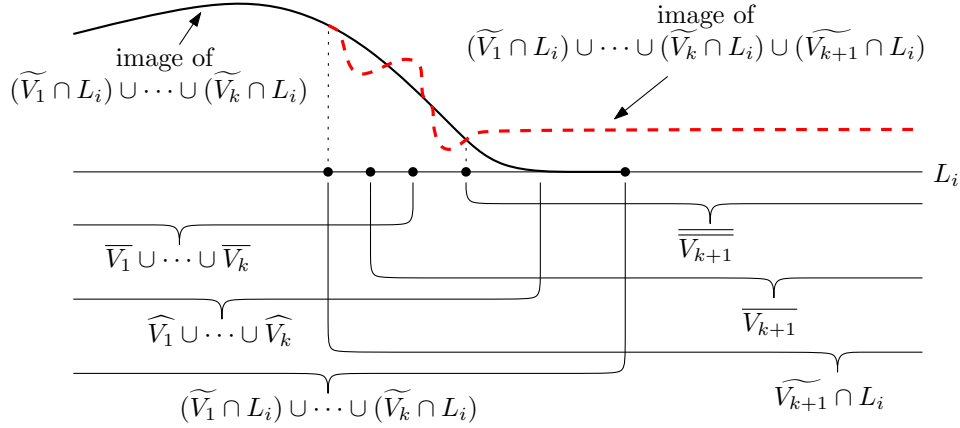}
\caption{Making $\overline{V_{x^{n-1}_{i,j,k+1}}}$ disjoint from $\Delta$ (where for each $m=1,...,k+1$ we write $\overline{V_{x^{n-1}_{i,j,m}}}=\overline{V_m}$, $\widehat{V_{x^{n-1}_{i,j,m}}}=\widehat{V_m}$, $\widetilde{V_{x^{n-1}_{i,j,m}}}=\widetilde{V_m}$, and also $\overline{\overline{V_{x^{n-1}_{i,j,k+1}}}}=\overline{\overline{V_{k+1}}}$ for simplicity).}
\label{fig:pushing_keeping_disjoint}
\end{center}
\end{figure}

Repeating the above process we can make the union $\overline{E_{i,1}^{n-1}} \cup \cdots \cup \overline{E_{i,l_i^{n-1}}^{n-1}}$ of the closures of all $(n-1)$-cells disjoint from $\Delta$ by pushing to a nearby isotropic copy in $\Upsilon^{-1}(N_Y(S) \cap X)$. Note that for any particular closure in the union if some part of it has been already pushed (this might happen if it has a common boundary part with another $(n-1)$-cell which has been pushed earlier), then we isotope it in an appropriate amount (as in the above discussion) so that previously pushed regions in the cell would not be moved back to $\Delta$.\\

Similarly, we can deal with the union $\{\overline{E_{i,1}^{K}} \cup \cdots \cup \overline{E_{i,l_i^{K}}^{K}}\}$ of all closed $K$-cells in $L_i$ (for $0 \leq K \leq n-2$) in the same way (under the assumption that all closed $(K+1)$-cells have been already made disjoint from $\Delta$. Note that we do not need to push any $K$-cell which appears as a part of the boundary of some $(K+1)$-cell(s) because such $K$-cells have been already made disjoint from $\Delta$ in the previous step. This process deforms the connected component $L_i$ to its image $L_i'$  which is isotropic and disjoint from $\Delta$. Repeating this for each connected component, we conclude that there exists an isotopy of the embeddings in $X$ from $L$ to its nearby isotropic copy $L'$ which is disjoint from $\Delta$. This isotopy is generated by contact vector fields $\textbf{Z}_{\mu_{x^{K}_{i,j,k}}}$ and compactly supported in the neighbourhood $$\widetilde{V}:=\bigcup_{0 \leq i \leq s} \quad \bigcup_{0 \leq K \leq n-1} \quad \bigcup_{0 \leq j \leq l_i^{K}} \quad \bigcup_{0 \leq k \leq m_{i,j}^{K}} \widetilde{V_{x^{K}_{i,j,k}}} \subset \Upsilon^{-1}(N_Y(S) \cap X)$$
Next, we want to extend this local isotopy of $L$ in $\Upsilon^{-1}(N_Y(S) \cap X) \times \{0\}$ to a local Legendrian isotopy of $S$ compactly supported in $\widetilde{V} \times [-1,1] \subset \Upsilon^{-1}(N_Y(S) \cap X) \times [-1,1] \subset T^*\Sigma^n \times \R$. (Here since $L$ is codimension $1$ submanifold of $S$, we can consider the tubular neighbourhood of $L$ in $S$ as the product $L \times [-1,1]$ such that $L$ corresponds to $L \times \{0\}$.) Consider the smooth cut-off function
$$f:\Upsilon^{-1}(N_Y(S) \cap X) \times [-1,1] \to \R, \quad f(x,t)=\mu(t)$$ where $\mu:[-1,1] \to \R$ is a smooth cut-off function which is equal to $1$ near $t=0$, and $0$ near $t=\pm 1$. Also let $\widetilde{\mu}_{x^{K}_{i,j,k}}:\Upsilon^{-1}(N_Y(S) \cap X) \times [-1,1] \to \R$ be the extension of $\mu_{x^{K}_{i,j,k}}$ given by $$\widetilde{\mu}_{x^{K}_{i,j,k}}(x,t)=\mu_{x^{K}_{i,j,k}}(x).$$ Denote by $\widetilde{\textbf{Z}}_{\mu_{x^{K}_{i,j,k}}}$ the contact vector field on $\Upsilon^{-1}(N_Y(S) \cap X) \times [-1,1]$ whose coresponding contact Hamiltonian (as in Lemma \ref{lem:contact_vector_field}) is equal to $f\widetilde{\mu}_{x^{K}_{i,j,k}} \textbf{H}_{x^{K}_{i,j,k}}$. Now we isotope $S$ to its nearby Legendrian copy $S'$ by applying the flow maps of $\widetilde{\textbf{Z}}_{\mu_{x^{K}_{i,j,k}}}$ in the same order and amount that we apply the flow maps of  $\textbf{Z}_{\mu_{x^{K}_{i,j,k}}}$ to isotope $L$ to $L'$. By construction, this isotopy is compactly supported in $\widetilde{V} \times [-1,1]$, and its restriction to $\Upsilon^{-1}(N_Y(S) \cap X) \times \{0\}$ is the isotopy taking $L$ to $L'$ (constructed above) as $\widetilde{\textbf{Z}}_{\mu_{x^{K}_{i,j,k}}}|_{\Upsilon^{-1}(N_Y(S) \cap X) \times \{0\}}=\textbf{Z}_{\mu_{x^{K}_{i,j,k}}}$. In particular, for the new Legendrian submanifold $S'$, its intersection $L'=S' \cap X$ is disjoint from the core $\Delta$ of $X$. \\

Finally, we note that by working on a small enough neighbourhood $N_Y(S)$ one can guarantee that $S'$ lies in any pre-given $\epsilon$-neighbourhood of $S$ in $Y$.
\end{proof}

By the last lemma we may assume that the transverse intersection $L=S \cap X$ of the Legendrian submanifold $S=\phi(\Sigma^n)$ is disjoint from the core $\Delta$ of $X=\varphi^{-1}(p)$. To finish the proof of Theorem \ref{thm:Missing_a_page}, we will construct an isotopy of Legendrian embeddings of $S$ which will be compactly supported in a neighbourhood of $X$ in $Y$ and will push $L$ completely outside $X$. To this end, we will first isotope $S$ along a contact vector field generated from $Z_3$ given in the list (\ref{eqn:Contact_fields}) above so that the part $L \cap Y_1$ of $L$ (recall $Y_1=\Sigma(X,h)_\rho$ is the mapping torus of the open book $\mathcal{OB}_{(X,h)}$) is completely pushed into the interior of $Y_2=B \times D^2$ (tubular neighbourhood of the binding $B=\partial X$). Then using $Z_1, Z_2$ of the list (\ref{eqn:Contact_fields}) we will isotope $S$ until $L \cap Y_2$ completely crosses the binding $B=B \times \{0\}$.

\begin{remark} \label{rem:for_the_last_part}
So far, when we write $X$ we meant the whole page (in particular, $\partial X$ was the binding $B$). However, for what follows  it is better to use the abstract open book desription as in the previous paragraph. Therefore, from now on $X$ will denote the complement of the collar neighbourhood of the binding $B$ in the corresponding page.
\end{remark}

From its construction the contact manifold $(Y_1=\Sigma(X,h)_\rho,\alpha|_{Y_1})$ is obtained as the quotient space of $(X \times \R, \beta+dz)$ using the equivalence relation $\sim_\rho$. Recall that we have $$Y_1=X \times \R / \sim_\rho \quad \textrm{where} \quad (x,z) \sim_\rho (h(x),z+\rho(x)\,).$$
%This identification depends on $\rho$ (which satisfies $h^*(\beta)-\beta=-d\rho$), and can be done
%in many different (but equivalent up to contactomorphism) ways by adding positive constants to $\rho$.
By translating (i.e., isotoping) $S$ along the Reeb direction $Z_4=R_\alpha$ (which corresponds to $\partial z$ on $X \times \R$), we may assume that $X=\varphi^{-1}(p)$ corresponds to $X \times \{0\}$ under this identification. Therefore, we can identify a neighbourhood of $X$ in $Y_1$ with $X \times [-a,a] \subset X \times \R$ for some real number $0<a<K$ where $K>0$ is a constant satisfying $$K < \rho(x), \quad \forall x \in X.$$ (Note that such $K$ exists since $X$ is compact and $\rho$ is a strictly positive continuous function). The contact form $\alpha$ on $Y$ is equal to $\beta + dz$ on $X \times [-a,a]$ where $Z_3$ takes the form $Z:=\chi+z\partial z$ as mentioned earlier. By choosing $a$ small enough, we may guarantee that the intersection $S \cap (X \times [-a,a])$ is equal to $L \times [-a,a]$ ($S$ and $X=X \times \{0\}$ intersect transversally), and also that $L \times [-a,a]$ is disjoint from $\Delta \times [-a,a]$ (this is because all cut-off functions which we used to isotope $S$ to $S'$ in the proof of Lemma \ref{lem:making_disjoint_from_core} are all equal to $1$ near $(L \times \{0\})\cap (\Delta \times \{0\})$.) \\

One can think of $Y_1$ slightly larger by considering $Y_2=B \times D^2$ slightly smaller. More precisely, let $N=B \times D \subset Y_2$ be a smaller neighbourhood of $B$ where $D \subset D^2$ is a smaller disk in $\R^2$ around the origin. By expanding each $(X,\beta)$ in $Y_1$ to a larger domain $(\tilde{X},\tilde{\beta})$, we get another decomposition $Y=Y_1' \cup N$ where $Y_1'=Y\setminus N$.
%$=\Sigma(\tilde{X},\tilde{h})$. (Here $\tilde{h} \in \textrm{Symp}(\tilde{X},\tilde{\beta})$ is obtained from $h$ by extending $h$ as identity over the region $\tilde{X} \setminus X$.)
Note that extending the above identification, a neighbourhood of $\tilde{X}$ in $Y_1'$ can be identified with $\tilde{X} \times [-a,a]$ ($\tilde{X}=\tilde{X} \times \{0\})$ on which the contact form $\alpha$ is given as $\alpha=\tilde{\beta} + dz$ and we have the extension $$\tilde{Z}=\tilde{\chi}+z\partial z$$ of $Z$ where $\tilde{\chi}$ is the $d \tilde{\beta}$-dual of $\tilde{\beta}$. We remark that $\tilde{Z}$ is contact with the contact Hamiltonian $H=z$. Let $\mu_1:Y \to \R$ be a smooth cut-off function such that $\mu_1 \equiv 1$ near $\tilde{X} \times\{0\}$ and $\mu_1 \equiv 0$ in the complement $Y \setminus \tilde{X} \times (-\epsilon,\epsilon)$ for some $0<\epsilon<a$ which will be determined later. Denote by $Z_{\mu_1}$ the contact vector field on $Y$ which corresponds to the contact Hamiltonian $\mu_1 H_3$.\\

Now we first push $S$ using the flow $Z_{\mu_1}^t$ until the image of $L$ lies completely in the interior $\textrm{int}(Y_2)$ of $Y_2$ as follows: First note that $Z_{\mu_1}=\tilde{\chi}$ on $\tilde{X}=\tilde{X} \times \{0\}$ (as $z=0$ there). Since $L$ is disjoint from the core of $\tilde{X}$ (which is the same as that of $X$), for every point $x \in L \cap \tilde{X}$ there exists a unique flow line of $\tilde{\chi}$ (i.e., of $Z_{\mu_1}|_{\tilde{X}}$) passing through $x$. All such flow lines reach the region $\tilde{X} \cap \textrm{int}(Y_2 \setminus N)$. Consider the set
$$A=\{t\in (0,\infty)\; |\; Z_{\mu_1}^t(x) \in \tilde{X} \cap \textrm{int}(Y_2 \setminus N) \quad \textrm{for all} \; x \in L \cap \tilde{X}\}.$$
Since $L \cap \tilde{X}$ is a compact, the set $A$ is non-empty. Choose a finite number $T>0$ from $A$. We isotope the Legendrian sphere $S \subset Y$ using the flow maps $\{Z_{\mu_1}^t\; | \; 0\leq t \leq T \}$. Indeed, by the construction of $Z_{\mu_1}$,  we only push the region $(L \cap \tilde{X}) \times (-\epsilon,\epsilon) \subset S$ during the isotopy. Observe that by choosing $\epsilon>0$ small enough, we can guarantee that the image $Z_{\mu_1}^T(L \times (-\epsilon,\epsilon))$ completely lies in $\tilde{X}\times [-a,a]$. Therefore, $Z_{\mu_1}^T(S)$ is an embedded Legendrian which is Legendrian isotopic to $S$ (see Figure \ref{fig:making_closer_to_binding}).\\

\begin{figure}[ht]
\begin{center}
\includegraphics{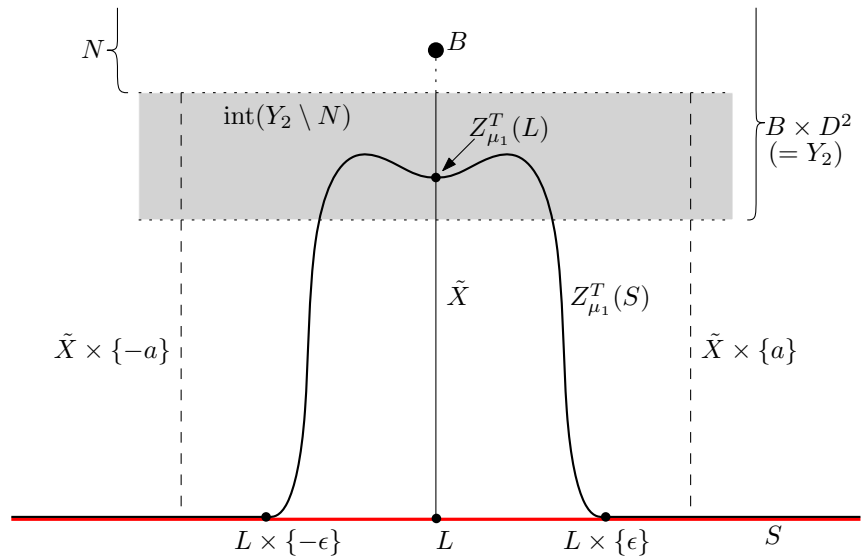}
\caption{Making $L$ closer to the binding via a Legendrian isotopy of $S$.}
\label{fig:making_closer_to_binding}
\end{center}
\end{figure}

By the choice of the flow parameter $T$ above, we know that $Z_{\mu_1}^T(L)$ is completely lie in the region $\textrm{int}(Y_2)$. For $0<r<1$ consider the disk $$D_r=\{(x,y)\in D^2 \mid x^2+y^2\leq r^2 \} \subset D^2$$ and the neighborhood $N_{r}=B \times D_r \subset \textrm{int}(B \times D^2)=\textrm{int}(Y_2)$ of the binding. Then after the above isotopy we know that $Z_{\mu_1}^T(L)  \subset N_{r}$ for some $r$. Let $(a_1,a_2) \in D^2$ be a point which corresponds to the angular coordinate $p \in S^1$ (recall that $\tilde{X}=\tilde{X} \times \{p\})$. Recall the contact vector fields $Z_1,Z_2$ from the list (\ref{eqn:Contact_fields}) and their contact Hamiltonians $H_1,H_2$. Then the vector field $Z'=a_1Z_1+a_2Z_2$ is also contact whose contact Hamiltonian is given by $H'=a_1H_1+a_2H_2$. Let $\mu_2:Y \to \R$ be a smooth cut-off function such that $\mu_2 \equiv 1$ near $N_r$ and $\mu_2 \equiv 0$ in the complement $Y \setminus Y_2$. Denote by $Z_{\mu_2}$ the contact vector field on $Y$ which corresponds to the contact Hamiltonian $\mu_2 H'$. Now we isotope $Z_{\mu_1}^T(S)$ using the flow maps of the contact vector field $Z_{\mu_2}$ until $Z_{\mu_1}^T(L)$ completely crosses the binding. Say for $T'>0$ the image of $Z_{\mu_1}^T(L)$ under the flow map $Z_{\mu_2}^{T'}$ completely crosses the binding. As a result, the closure of the page $\varphi^{-1}(p)$ of the open book $\mathcal{OB}_{(X,h)}$ is disjoint from the final image $S':=Z_{\mu_2}^{T'}(Z_{\mu_1}^T(S))$ which is Legendrian isotopic to $S=\phi(\Sigma^n)$ as claimed.
This finishes the proof of Theorem \ref{thm:Missing_a_page}.
\end{proof}

%-----------------------------------------------------------------------------------------------
%-----------------------------------------------------------------------------------------------
%-----------------------------------------------------------------------------------------------

\begin{proof}[\textbf{Proof of Theorem \ref{thm:Main_Theorem_1}}]
Let $\mathcal{OB}$ be an open book decomposition carrying a contact structure $\xi$ on a (closed) manifold $M$ of dimension $2n+1\geq 3$. In particular, there exists a contact form $\alpha$ for $\xi$ such that $\alpha$ restricts to a Liouville form on every page of $\mathcal{OB}$. By assumption, for any page $X$ of $\mathcal{OB}$, the restriction $\alpha|_X$ is, indeed, the underlying Liouville form of a Weinstein structure on $X$.\\

Pick a page $X$ equipped with the Liouville form $\beta:=\alpha|_X$ which is, by assumption, the underlying Liouville form of a Weinstein structure on $X$. Denote by $h \in \textrm{Symp}(X,d\beta)$ the monodromy of $\mathcal{OB}$ and consider the manifold $M(X,h)$ equipped with the contact structure $\xi_\beta=\textrm{Ker}(\alpha_\beta)$ where $\alpha_\beta$ is the contact form on $M(X,h)$ constructed as in Remark \ref{rem:constructing_contact_form}. Let $$\xi_X:=(\Upsilon_X^{-1})_*(\xi_\beta)$$ be the contact structure on $M$ obtained by pushing forward $\xi_\beta$ using the inverse of the identification map $\Upsilon_X:M \to M(X,h)$. As explained at the beginning of this section, $\Upsilon_X$ respects the fibration maps on $M\setminus B$ and $M(X,h)\setminus (\partial X \times \{0\})$  associated to the open books $\mathcal{OB}$ and $\mathcal{OB}_{(X,h)}$, respectively. Therefore, we have a contactomorphism  $$\Upsilon_X:(M,\xi_X) \to (M(X,h),\xi_\beta)$$ mapping pages of $\mathcal{OB}$ to those of $\mathcal{OB}_{(X,h)}$. \\

Next, observe that $\xi$ and $\xi_X$ are supported by the same open book $\mathcal{OB}$ (by construction of $\xi_X$), so we know, by Giroux's work, that there exists an isotopy $\xi_t$ ($t \in [0,1]$) of contact structures on $M$ connecting $\xi_0=\xi$ and $\xi_1=\xi_X$. Then Gray's Stability (see, for instance, Theorem 2.2.2 of \cite{Geiges}) implies that there is a diffeotopy $$\Xi_t: M \to M, \quad t \in [0,1],$$ such that $(\Xi_t)_*(\xi_0)=\xi_t$ for each $t$. In particular, $(\Xi_1)_*(\xi)=\xi_X$, and hence we obtain two contactomorphisms $$f_X:=\Xi_1:(M,\xi) \to (M,\xi_X), \quad \Upsilon_X \circ f_X:(M,\xi) \to (M(X,h),\xi_\beta).$$

\vspace{.3cm}
Suppose now we are given a Legendrian embedding $\psi: \Sigma^n \hookrightarrow (M,\xi)$ of a compact  Legendrian submanifold $L=\psi(\Sigma^n)$. By pushing forward $\psi$ using the above contactomorphisms, we obtain two Legendrian embeddings $$f_X \circ \psi:\Sigma^n \hookrightarrow(M,\xi_X), \quad \phi:=\Upsilon_X \circ f_X \circ \psi:\Sigma^n \hookrightarrow (M(X,h),\xi_\beta).$$

\vspace{.3cm}
We set $S:=\phi(\Sigma^n)$. By Theorem \ref{thm:Missing_a_page} we have a smooth $1$-parameter family $$\Phi_t:\Sigma^n \hookrightarrow (M(X,h),\xi_\beta), \quad t \in [0,1]$$ of Legendrian embeddings such that $\Phi_0(\Sigma^n)=S$ and the Legendrian submanifold $S\,':=\Phi_1(\Sigma^n)$ is disjoint from the closure of a page of the open book $\mathcal{OB}_{(X,h)}$ on $M(X,h)$ associated to $(X,h)$. By composing $\Phi_t$ with the contactomorphism $\Upsilon_X^{-1}$, we obtain a smooth $1$-parameter family $$\Psi_t:=\Upsilon_X^{-1} \circ \Phi_t:\Sigma^n \hookrightarrow (M,\xi_X), \quad t \in [0,1]$$ of Legendrian embeddings such that $\Psi_0(\Sigma^n)=f_X(L)$. Finally, using this isotopy and the fact that $\Upsilon_X^{-1}$ is mapping pages of $\mathcal{OB}_{(X,h)}$ to the corresponding pages of $\mathcal{OB}$, we conclude that the Legendrian submanifold  $$\Psi_1(\Sigma^n)=\Upsilon_X^{-1}(S\,')$$ is Legendrian isotopic to $f_X(L)$ and disjoint from a page of the open book $\mathcal{OB}$. Thus, setting $f:=f_X$ and $\xi':=\xi_X$ finishes the proof of Theorem \ref{thm:Main_Theorem_1}.
\end{proof}

%-----------------------------------------------------------------------------------------------
%-----------------------------------------------------------------------------------------------
%-----------------------------------------------------------------------------------------------

\addcontentsline{toc}{chapter}{\textsc{References}}

\end{document}